\documentclass[submission]{FPSAC_noheader}

\newcommand{\beqa}{\begin{eqnarray}}
\newcommand{\eeqa}{\end{eqnarray}}

\newtheorem{Theorem}{Theorem}[section]
\newtheorem{theorem}[Theorem]{Theorem}

\newtheorem{lemma}[Theorem]{Lemma}

\newtheorem{corollary}[Theorem]{Corollary}

\newtheorem{example}[Theorem]{Example}

\newtheorem{definition}[Theorem]{Definition}

\newcommand{\pf}{\mathrm{Pf}} 

\newcommand{\beq}{\begin{equation}}
\newcommand{\eeq}{\end{equation}}

\def\NN{{\mathbb N}}

\def\NN{{\mathbb N}}

\usepackage{lipsum}

\title[Pfaffians and nonintersecting paths in graphs with cycles]{Pfaffians and nonintersecting paths in graphs with cycles: Grassmann algebra methods}

\author[S. Carrozza and A. Tanasa]{S. Carrozza\thanks{scarrozza@perimeterinstitute.ca}\addressmark{1} 
\and A. Tanasa\thanks{ntanasa@u-bordeaux.fr}\addressmark{2}
}

\address{\addressmark{1}Perimeter Institute for Theoretical Physics, Waterloo, Canada \\ 
\addressmark{2}Univ. Bordeaux, Talence, France, EU\\
H. Hulubei National Inst. Phys. Nucl. Engineering, M\u{a}gurele, Romania, EU\\
I. U. F., Paris, France, EU
}

\abstract{After recalling the definition of Grassmann algebra and elements of Grass\-mann--Berezin calculus, we use the expression of Pfaffians as Grassmann integrals to generalize a series of formulas relating generating functions of paths in digraphs to Pfaffians. We start with the celebrated Lindstr\"om-Gessel-Viennot formula, which we derive in the general case of a graph with cycles. We then make further use of Grassmann algebraic tools to prove a generalization of the results of (Stembridge 1990). Our results, which are applicable to graphs with cycles, are formulated in terms of systems of nonintersecting paths and nonintersecting cycles in digraphs.
}

\keywords{Grassmann algebra, Grassmann-Berezin calculus, Pfaffians, nonintersecting paths, nonintersecting cycles.}

\usepackage[backend=bibtex]{biblatex}
\begin{document}

\maketitle

\section{Introduction and motivation}
\label{sec:in}

Grassmann algebra (see, for example \cite{feldman} for a general reference) is extensively used in mathematical physics to describe fermions, which are elementary particles obeying the so-called Fermi-Dirac statistics. The generators of the algebra are called Grassmann variables and they obey anti-commuting relations (unlike real or complex variables, which obey commuting relations, and which are used in mathematical physics to describe bosons). 
One can then define the so-called Grassmann integral, which is a linear form on the Grassmann algebra. 

The rules of Grassmann calculus lead to very 
interesting developments outside the original field of mathematical physics. Thus, in combinatorics, Grassmann algebra 
is used to derive elegant expressions for determinants and Pfaffians.
These expressions have already been used to prove
the contraction-deletion formula of the Tutte and Bollob\'as-Riordan polynomials \cite{KRTW}, matrix-tree-like theorems \cite{malek}, 
generalizations of the classical Cauchy-Binet formula 
\cite{css}
 and so on.


In \cite{gascom}, these techniques were used to give a new proof of the Lindstr\"{o}m-Gessel-Viennot (LGV) formula for graphs with cycles (formula first found by means of purely combinatorial techniques in \cite{talaska} -- the interested reader should also consult \cite{lalonde}
for a combinatorial proof of Jacobi's equality relating a cofactor of a matrix with the complementary factor of its inverse).

In this paper, we use Grassmann algebra to give, first,
a reformulation of the LGV lemma proof already proposed in \cite{gascom}.
Our proof is identical to that of \cite{pzinn}, but we show that it generalizes to graphs with cycles. We then use the same Grassmann techniques to give new proofs of Stembridge's identities relating appropriate graph Pfaffians to sum over nonintersecting paths - see \cite{stem}.
Our results go further than the ones of \cite{stem}, because Grassmann algebra techniques naturally extend (without any cost!) to graphs with cycles. 
We thus obtain, instead of sums over nonintersecting paths, sums over nonintersecting paths and nonintersecting cycles. In the last section, we give a generalization of all these results.

It is worth mentioning here that our proofs do not require any clever involution argument (as it is the case for the "classical" proofs of \cite{GV}, \cite{talaska} or \cite{stem}). The vanishing contribution of intersecting paths and cycles is an immediate consequence of the basic properties of Grassmann algebras. 

\section{Grassmann algebras and Grassmann-Berezin calculus}

\subsection{The Grassmann algebra}

Let $R$ be a unital commutative ring such that $\mathbb{Q}\subseteq R$, and let $\chi_1,\ldots,\chi_m$ ($m\in\NN)$ be a collection of letters.

\begin{definition}
The {\bf Grassmann algebra} $R[\chi_1,\ldots, \chi_m]$, or simply $R[\chi]$, is the quotient of the free noncommutative $R-$algebra with generators  $\chi_1,\ldots,\chi_m$, by the two-sided ideal generated by the expressions  $\chi_i \chi_j + \chi_j \chi_i$, for all  
$i,  j=1,\ldots , m$.
\end{definition}

Thus, the generators of the Grassmann algebra $\chi_1, ..., \chi_m$, also called
Grassmann (or anti-commuting) variables 
satisfy anticommutation relations: 
\beqa
\label{defg}
 \chi_i \chi_j +\chi_j \chi_i=0, \quad \forall  i,  j=1,\ldots , m \, .
\eeqa
As an immediate consequence (since $2$ is invertible in $R$), one has the following crucial identity (which encodes Pauli's exclusion principle in the physics literature):
\beqa
\label{defg2}
 \chi_i^2= 0, \quad \forall  i=1,\ldots , m \, .
\eeqa

One can then easily prove that a general element of the Grassmann algebra 
can be uniquely written as
\beq
\label{deff}
f(\chi) = \sum_{n=0}^{m} \sum_{1 \leq i_1< \ldots \, <i_n \leq m} a_{i_1\ldots i_n} \chi_{i_1} \ldots \chi_{i_n} \, ,
\eeq
with $a_{i_1 \ldots i_n}\in R$.

The multiplication rule for Grassmann monomials is
\begin{equation}\label{multiplication}
(\chi_{i_1} \ldots \chi_{i_n}) (\chi_{j_1} \ldots \chi_{j_p}) = \left\{ 
\begin{split} &0 \qquad &\mathrm{if} \quad \{ i_1, \ldots , i_n \} \cap \{ j_1, \ldots , j_p \} \neq \emptyset \\
 &\mathrm{sgn}(k) \, \chi_{k_1} \ldots \chi_{k_{n+p}} & \mathrm{otherwise} 
\end{split}\right.\, ,
\end{equation}
where $1 \leq i_1< \ldots \, <i_n \leq m$, $1 \leq j_1< \ldots \, <j_n \leq m$, and $k = (k_1, \ldots , k_{n+p})$ is the permutation of $(i_1, \ldots , i_n, j_1 , \ldots , j_{p})$ such that $k_1 < \ldots < k_{n+p}$. The multiplication rule for general Grassmann algebra elements follows by distributivity. 
Note that Grassmann variables and $R-$variables commute.

The exponential of a Grassmann function $f$ is defined by the usual formula:
\begin{equation}
\label{expg}
   \exp({f(\chi)}) := \sum_{p = 0}^{+ \infty} \frac{1}{p!}\left( f(\chi)\right)^p\,.
\end{equation}
Note that this is a polynomial in $\chi$ (as a consequence of \eqref{defg} and of the fact that the number of generators is finite).
For example, one has:
\beq
\exp ({\chi_{i_1} \ldots \chi_{i_n}}) = 1 + \chi_{i_1} \ldots \chi_{i_n}.
\eeq

\subsection{Grassmann-Berezin calculus; Pfaffians as Grassmann integrals}

Following \cite{berezin},  we define the {\bf Grassmann integral} $\int d\chi \equiv \int d\chi_m \ldots d\chi_1 $ as the unique linear map from 
$R[\chi_1, \ldots, \chi_m]$ to $R$ such that
\beqa
\label{intg}
  \int d\chi \, \chi_1 \ldots \chi_m = 1\,,  \qquad \mathrm{and}\quad \int d\chi \, \chi_{i_1} \ldots \chi_{i_n} = 0 \quad \mathrm{if} \quad n < m\,.
  \eeqa
Thus, for $f\in R[\chi_1 \ldots \chi_m]$ written as in \eqref{deff}, one has:
\beqa
\int d \chi \, f(\chi) = a_{12\ldots m}\,.
\eeqa

\begin{example}
\label{expfaff}
Let  $\chi_1$ and $\chi_2$ be two independent Grassmann variables. One has:
\beq\label{rule_int}
\int d\chi_1 d\chi_2 \exp (-\frac 12  \chi_1 a \chi_2+ \frac 12 \chi_2 a \chi_1)= 
\int d\chi_1 d\chi_2 (1- \frac 12  \chi_1 a \chi_2+ \frac 12 \chi_2 a \chi_1)=a \, .
\eeq
\end{example}




If $m$ is an even integer and $A$ is an $m\times m$ skew-symmetric matrix with coefficients in $R$, then its Pfaffian 
writes as: 
\beqa
\label{pfaffeq}
\pf (A)=\int d\chi_1 \ldots d \chi_m \exp({-\frac 12 \sum_{i,j=1}^m \chi_i A_{ij} \chi_j}) \, .
\eeqa
Example \ref{expfaff} above is thus an illustration of this formula for the matrix $A=\begin{pmatrix}
0 & - a \\ a & 0 \end{pmatrix}$.

Finally, let now $M$ be an $n-$dimensional square matrix, again with entries in $R$, and 
consider $2n$ Grassmann variables $\bar \chi_i,\chi_i$, $i=1,\ldots, n$. Note that the conjugate notation has nothing to do with complex conjugation.
One has:
\beqa\label{det}
\det M = \int d\bar \chi_N d\chi_N \ldots d\bar \chi_1 d\chi_1 \,  \exp\left(-\sum_{i,j=1}^N \bar\chi_i M_{ij} \chi_j  \right) \, .
\eeqa
Formulas \eqref{pfaffeq} and \eqref{det} above are of course compatible with the standard relation between determinants and Pfaffians:
$\pf\begin{pmatrix}
0 & M \\
- {}^tM & 0
\end{pmatrix} = (-1)^{\frac{N(N-1)}{2}} \det M \, .
$


\section{On Grassmann Gaussian measures}

In this section we present a preliminary identity, which we will use extensively in the rest of the paper.


Let $S$ be a $2n \times 2n$ skew-symmetric, invertible matrix. 
Since the Pfaffian is a square-root of the determinant, this implies in particular that $
\int d \chi \exp\left( - \frac{1}{2} \sum_{i,j = 1}^{2n} \chi_i S_{ij}^{-1} \chi_j\right) \neq 0\, .
$
Following \cite{feldman}, one defines a notion of (normalized) Gaussian measure associated to $S$:

\begin{definition}
The {\bf Gaussian measure} $d \mu_S$ associated to $S$ is given by the following formula:
\beq\label{gauss}
d \mu_S (\chi) :=  d\chi\, \frac{\exp\left( - \frac{1}{2} \sum_{i,j = 1}^{2n} \chi_i S_{ij}^{-1} \chi_j\right)}{\int d \eta \exp\left( - \frac{1}{2} \sum_{i,j = 1}^{2n} \eta_i S_{ij}^{-1} \eta_j\right)} \, .
\eeq
\end{definition}

Let us introduce the following notation: for any square matrix $M$ of size $p$, and any subsets $\mathcal{A}, \mathcal{B} \subset \{1, \ldots, p\}$, we will denote by $M_{\mathcal{A} \mathcal{B}}$ the matrix obtained from $M$ by deletion of all lines (resp. columns) which are not in $\mathcal{A}$ (resp. $\mathcal{B}$).  

The moments of Grassmann Gaussian measures are related to Pfaffians in the following way:
\begin{lemma}
(Proposition $I. 19$ of \cite{feldman})
\label{wick}
Let $S$ be a $2n \times 2n$ skew-symmetric and invertible matrix. For any even $k \geq 2$, and for any subset $\mathcal{I} = \{i_1, \ldots , i_k \}$ with $1\leq i_1 < \ldots < i_k \leq 2n$, one has:
\beq\label{wickeq}
\int d\mu_S (\chi) \chi_{i_1} \ldots \chi_{i_k} = \pf \left( S_{\mathcal{I}\mathcal{I}} \right)\,.
\eeq
\end{lemma}

As already noticed in \cite{feldman}, 
one may extend the notion of Gaussian measure to non-invertible skew-symmetric matrices $S$, in which case $d \mu_S$ is \emph{defined} by \eqref{wick}.



\section{LGV formula for graphs with cycles}

In this section, we present a simplified version of the proof given in \cite{gascom} of the LGV lemma for graphs with cycles. 

Let $G$ be a digraph (or a directed graph) which allows loops or multiple edges.
Let $V =\{ v_1, \ldots, v_n\}$ be the set of vertices of the digraph, which we equip with an order relation $<$ such that $v_1 < \ldots < v_n$\footnote{To simplify notations, we will often identify the ordered set $V$ with $\{ 1, \ldots n\}$ itself.}.
To each edge $e$ of the digraph, we assign a {\bf weight} denoted by $w_e$. One further assumes that the variables $w_e$ thus defined are commuting variables. 

Let us consider a collection of edges $P = (e_1, e_2, \ldots , e_k)$ such that one can reach a vertex $v'$ from a vertex $v$ by successively traversing the edges of the collection $P$ in the specified order.
We then say that $P$ is a {\bf path} in the digraph, from the vertex $v$ to the  vertex $v'$.

 
Let the product 
$\mathrm{wt}(P):=\prod_{k=1}^{m} w_{e_k}$ 
denote the {\bf weight} of a given path $P=(e_1,\ldots , e_k)$. We then construct the {\bf weight path matrix} 
$M_{ij}:=\underset{{P:v_i \rightsquigarrow v_j}}{\sum} \mathrm{wt}(P)
$, where $1\leq i,j\leq n$ and 
the sum runs over paths $P$ connecting the vertex $v_i$ to the vertex $v_j$.
Note that the entries of the matrix $M$ are considered as formal power series in the edge weights $w_e$. Matrix equations such as $M = (1 - A)^{-1}$, where $A$ is the weighted adjacency matrix\footnote{For any $1 \leq i,j \leq n$, $A_{ij} := \underset{e:v_i \rightsquigarrow v_j}{\sum} w_e$, where the sum runs over edges connecting $v_i$ to $v_j$.}
, are also meant in a formal power series sense. 

A {\bf cycle} is defined as a path from a vertex $v$ to itself (up to a change of source vertex $v$ along the path). 
The set of all possible collections of self-avoiding and pairwise vertex-disjoint cycles, including the empty collection, is denoted by $\mathcal{C}$. 
Let the weight and sign of $\mathbf{C} = (C_1, \ldots, C_k) \in \mathcal{C}$ be given by the following formulas:
\beq
\mathrm{wt}(\mathbf{C}):= \prod_{i=1}^k \mathrm{wt}(C_i)\,, \qquad \mathrm{} \qquad \mathrm{sgn}(\mathbf{C}):= (-1)^k \, .
\eeq
We further set $\mathrm{wt}(\mathbf{C}) =1$ and $ \mathrm{sgn}(\mathbf{C}) =1$ for the empty collection $\mathbf{C} = \emptyset$.

One considers now subsets of vertices 
${\cal A} = \{ a_1, \ldots , a_p \}$ and ${\cal B} = \{ b_1, \ldots , b_p \}$, with $a_1 < \ldots < a_p$ and $b_1 < \ldots < b_p$.
A collection of paths $\mathbf{P} = (P_1, \ldots, P_p)$, where the path $P_i$ connects the vertex $a_i$ to the vertex $b_{\sigma_{\mathbf{P}} (i)}$ (where $\sigma_{\mathbf{P}}$ is some permutation), 
is called a {\bf $p-$path} from the subset $\cal A$ to the subset $\cal B$. 
Let 
\beq
\mathrm{wt}(\mathbf{P}):= \prod_{i=1}^k \mathrm{wt}(P_i)\,, \qquad \mathrm{} \qquad \mathrm{sgn}(\mathbf{P}):= \mathrm{sgn}(\sigma_{\mathbf{P}}) \, ,
\eeq
denote the weight and resp. the sign of the $p-$path $\mathbf{P}$. One says that  
$\mathbf{P}$ is {\bf self-avoiding} if: each path $P_i$ is self-avoiding and if, for each pair of paths $P_i$ and $P_j$ with $i\neq j$, $P_i$ and $P_j$ are vertex-disjoint. 

Let us now give the following definition:
\begin{definition}
\label{flow}
A {\bf self-avoiding flow} from $\cal A$ to $\cal B$ is a pair $(\mathbf{P}, \mathbf{C})$ such that  
$\mathbf{C} \in \mathcal{C}$, $\mathbf{P}$
is a self-avoiding $p$-path from $\cal A$ to $\cal B$, 
  and $\mathbf{P}$ and $\mathbf{C}$ are vertex disjoint. The set of self-avoiding flows from $\cal A$ to $\cal B$ is denoted by $\mathcal{F}_{{\cal A},{\cal B}}$. We furthermore identify $\mathcal{F}_{\emptyset,\emptyset}$ with $\mathcal{C}$.
\end{definition}

Given $2n$ independent Grassmann variables $\theta_i$ and $\bar\theta_i$ with $i$ running from $1$ to $n$, we introduce the notations
\beqa
[d \theta \cdot d \bar\theta] := d \theta_n d \bar\theta_n \ldots d \theta_1 d \bar\theta_1\,, \qquad  [\bar\theta_{\mathcal{B}} \cdot \theta_{\mathcal{A}} ] := \bar\theta_{b_1} \theta_{a_1} \ldots \bar\theta_{b_p} \theta_{a_p}\,,
\eeqa
where $\mathcal{A}$ and $\mathcal{B}$ are sets as above. 

One then has: 
\begin{lemma}\label{lemma:paths}
Let $\mathcal{A}$ and $\mathcal{B}$ be two subsets of $\{ 1, \ldots, n\}$ with cardinal $p$. Denote by $a_1 < \ldots < a_p$ the elements of $\mathcal{A}$ and $b_1 < \ldots < b_p$ the elements of $\mathcal{B}$. Then:
\beqa\label{basic_int}
\int [d \theta \cdot d \bar\theta]  \, [\bar\theta_{\mathcal{B}} \cdot \theta_{\mathcal{A}} ] \,\exp\left( \sum_{i,j = 1}^N \bar\theta_i M^{-1}_{ij} \theta_j \right) = \underset{(\mathbf{P}, \mathbf{C})\in \mathcal{F}_{{\cal A},{\cal B}}}{\sum} \mathrm{sgn}(\mathbf{P})\mathrm{wt}(\mathbf{P}) \, \mathrm{sgn}(\mathbf{C})\mathrm{wt}(\mathbf{C}) \, .
\eeqa
\end{lemma}
\begin{proof}
Let us give the main idea of the proof.
As noticed above, $M^{-1} = 1 - A$, and one can rewrite the exponential appearing in the integrand as a product of terms of the form $(1 + \bar\theta_i \theta_i)$ and $(1 - \bar\theta_i A_{ij} \theta_j)$. After expansion of this product, the monomials which do not vanish after Grassmann integration must contain each Grassmann variable once, and only once. This implies that they are indexed by elements of $\mathcal{F}_{{\cal A},{\cal B}}$. One furthermore checks that each flow is weighted as in the right-hand side of formula \eqref{basic_int}.
\end{proof}

The LGV formula for a digraph with cycles is:
\begin{theorem}
\label{GV:thm} (Theorem $2.5$ of \cite{talaska})
One has:
\begin{equation}\label{gv}
    \det (M_{{\cal A} {\cal B}}) =\frac{\underset{(\mathbf{P}, \mathbf{C})\in \mathcal{F}_{{\cal A},{\cal B}}}{\sum} \mathrm{sgn}(\mathbf{P})\mathrm{wt}(\mathbf{P}) \, \mathrm{sgn}(\mathbf{C})\mathrm{wt}(\mathbf{C})}{\underset{\mathbf{C} \in \mathcal{C}}{\sum} \mathrm{sgn}(\mathbf{C}) \mathrm{wt}(\mathbf{C})} \, .
\end{equation}
In particular, if $G$ is acyclic, one recovers the LGV Lemma \cite{Li}, \cite{GV}.
\end{theorem}
\begin{proof}
Let us define the skew-symmetric matrix $S$ (and its inverse) as:
\beqa
\label{s1}
S=
\begin{pmatrix}
0 & - M \\
 {}^t M & 0
\end{pmatrix}\,, \qquad 
S^{-1}=
\begin{pmatrix}
0 &  {}^t M^{-1} \\
- M^{-1} & 0
\end{pmatrix}\,.
\eeqa
Define $\overline{\mathcal{B}} := \{ n+ b \vert b \in \mathcal{B}\}$. We now use Lemma \ref{wick} with the matrix \eqref{s1} and $\mathcal{I} = \mathcal{A} \cup \overline{\mathcal{B}}$. Using the explicit expression of $S$, one finds
\beqa
S_{\mathcal{I} \mathcal{I}} = \begin{pmatrix}
0 & - M_{\mathcal{A} \mathcal{B}} \\
 {}^t M_{\mathcal{A} \mathcal{B}} & 0
\end{pmatrix}\,,
\eeqa
which implies that 
$\pf\left( S_{\mathcal{I}\mathcal{I}}\right) = (-1)^{\frac{p(p+1)}{2}} \det\left( M_{\mathcal{A} \mathcal{B}} \right).$
One then has:
\beqa
\det\left( M_{\mathcal{A} \mathcal{B}} \right) = (-1)^{\frac{p(p+1)}{2}} \int d \mu_S (\chi) \left( \chi_{a_1} \ldots \chi_{a_p}\right) \left( \chi_{n+b_1} \ldots \chi_{n+b_p}\right)\,.
\eeqa 
Using the explicit form \eqref{gauss} of the Gaussian measure, we can rewrite this identity as
\begin{align}
\det (M_{{\cal A} {\cal B}})
&= \frac{1}{Z}
\int [d\theta \cdot d \bar\theta ] \,
[\bar\theta_{\mathcal{B}}  \cdot \theta_{\mathcal{A}} ]
\exp \left( \sum_{i,j=1}^n \bar\theta_i M^{-1}_{ij} \theta_j\right)\, ,
\end{align}
where we have introduced the Grassmann variables $\theta_i = \chi_i$ and $\bar\theta_i = \chi_{n+i}$ ($i \in\{1, \ldots, n\}$), and the normalization 
\beqa
Z=\int [d\theta \cdot d \bar\theta ] \,
\exp \left( \sum_{i,j=1}^n \bar\theta_i M^{-1}_{ij} \theta_j\right)\,.
\eeqa
To conclude the proof, we apply Lemma \ref{lemma:paths} to both the numerator and denominator of $\det (M_{{\cal A} {\cal B}})$. The evaluation of $Z$ reduces to a sum over non-intersecting cycles, while the numerator contains all non-intersecting flows in $\mathcal{F}_{{\cal A},{\cal B}}$.
\end{proof}

\section{Stembridge's formulas for graphs with cycles}

In this section, we
consider the situation in which the paths start from a specified set $\cal A$ of $p$ vertices (as above) but their endpoints are not fixed. They are only restricted to lie within the graph region ${\cal I}\subset \{ 1, \ldots, n\}$, where the cardinal of $\mathcal{I}$ can be larger than $p$. Moreover, we assume the cardinal of $\mathcal{A}$ to be even: $p = 2 m$.

Following \cite{stem}, let us define the $n \times n$ matrix 
$Q^{\mathcal{I}}$ whose entries are:
\beqa
Q^{\mathcal{I}}_{ij}:= \sum_{k, l\in \mathcal{I}\vert k<l} \left( M_{ik} M_{jl} - M_{il} M_{jk} \right)\,.
\eeqa
Let us furthermore denote by 
\beqa
\mathcal{F}_{{\cal A}}^{\mathcal{I}}:= \bigcup_{\mathcal{B} \subset \mathcal{I}, \vert\mathcal{B}\vert = \vert\mathcal{A}\vert} \mathcal{F}_{{\cal A},\mathcal{B}}
\eeqa
the set of self-avoiding flows from $\cal A$ to $\mathcal{I}$ (defined as in Definition \ref{flow}). In this definition, the union over subsets $\mathcal{B} \subset \mathcal{I}$ is disjoint; we can thus naturally extend the weight and sign functions previously introduced to $\mathcal{F}_{{\cal A}}^{\mathcal{I}}$. Note that the sign is always computed relative to the natural order on the set of integers $\mathcal{A}$ and $\mathcal{B}$. 

We can now state the main result of this section:

\begin{theorem}
\label{thm31}
Let $\mathcal{A}$ and $\mathcal{I}$ be subsets of $\{ 1, \ldots , n\}$ such that $\vert\mathcal{A}\vert$ is even. One has: \beqa
\pf \left( Q^{\mathcal{I}}_{\mathcal{A} \mathcal{A}}\right) = 
\frac{\underset{(\mathbf{P}, \mathbf{C})\in \mathcal{F}_{{\cal A}}^{\mathcal{I}}}{\sum} \mathrm{sgn}(\mathbf{P})\mathrm{wt}(\mathbf{P}) \, \mathrm{sgn}(\mathbf{C})\mathrm{wt}(\mathbf{C})}{\underset{\mathbf{C} \in \mathcal{C}}{\sum} \mathrm{sgn}(\mathbf{C}) \mathrm{wt}(\mathbf{C})} \, .
\eeqa
\end{theorem}
\begin{proof}
As before, we denote by $a_1 < \ldots < a_{2m}$ the elements of $\cal A$.
 Let us also introduce the $2n \times 2n$ skew-symmetric and invertible matrix
\beqa
\label{s31}
S^{-1}=
\begin{pmatrix}
0 &  {}^t M^{-1} \\
- M^{-1} & B
\end{pmatrix},
\eeqa
where the $n-$dimensional and skew-symmetric square matrix $B$ is defined by: $B_{ij}=1$ if $i, j\in \mathcal{I}$ and $i<j$, 
 $B_{ij}=-1$ if $i, j\in \mathcal{I}$ and $ i> j$, and $B_{ij}=0$ otherwise. We have then:
 \beq
 S = 
\begin{pmatrix}
M B {}^t M & - M \\
 {}^t M & 0
\end{pmatrix}
=
\begin{pmatrix}
Q^{\mathcal{I}} & - M \\
 {}^t M & 0
\end{pmatrix}  \,  .\eeq
As in the previous section, we make use of Lemma \ref{wick}, which now implies:
\beqa
\pf\left( Q^{\mathcal{I}}_{\mathcal{A}\mathcal{A}}\right) = \pf\left( S_{\mathcal{A}\mathcal{A}}\right) = \int d \mu_{S} (\chi) \chi_{a_1} \ldots \chi_{a_{2m}} \,.
\eeqa 
Using the Grassmann variable conjugate notation $\{ \theta_i , \bar\theta_i \}$, we can rewrite this equality in the form
\beqa\label{action}
\pf\left( Q^{\mathcal{I}}_{\mathcal{A}\mathcal{A}}\right) = \frac{1}{Z} \int [d\theta \cdot d \bar\theta ] \, \theta_{a_1} \ldots \theta_{a_{2m}} \, \exp\left( \sum_{i}^n \bar\theta_i \theta_i - \sum_{i,j = 1}^n \bar\theta_i A_{ij} \theta_j - \sum_{i,j\in \mathcal{I}\vert i<j} \bar\theta_i \bar\theta_j \right)\, ,
\eeqa
where:
\beqa
Z = \int [d\theta \cdot d \bar\theta ] \, \exp\left( \sum_{i}^n \bar\theta_i \theta_i - \sum_{i,j = 1}^n \bar\theta_i A_{ij} \theta_j - \sum_{i,j\in \mathcal{I}\vert i<j} \bar\theta_i \bar\theta_j \right) \,.
\eeqa
It is easy to see that, when expanding the integrand of $Z$ into monomials, all terms containing a factor $\bar\theta_i \bar\theta_j$ coming from the third sum in the exponential will vanish, since such terms will contain strictly more $\bar\theta$ variables than $\theta$ variables. Hence the third sum in the exponential does not contribute, and $Z$ is again a sum over non-intersecting cycles. As for the numerator of \eqref{action}, it can be expanded as a sum over subsets $\mathcal{B} \subset \mathcal{I}$ of size $2m$
\beq
\sum_{\mathcal{B} \subset \mathcal{I}, \vert \mathcal{B} \vert = 2m } (-1)^m \int [d\theta \cdot d \bar\theta ] \, \theta_{a_1} \ldots \theta_{a_{2m}} \, F_{\mathcal{B}} (\bar\theta) \, \exp\left(\sum_{i,j = 1}^n \bar\theta_i M^{-1}_{ij} \theta_j \right)\, ,
\eeq
where $F_{\mathcal{B}} (\bar\theta)$ is a weighted sum of monomials. More precisely, denoting by $b_1 < \ldots < b_{2m}$ the elements of $\mathcal{B}$, we have
\beq
F_{\mathcal{B}} (\bar\theta) = \sum_{(i_1 i_2) \ldots  (i_{2m-1} i_m)} \prod_{k = 1}^m \bar\theta_{b_{i_{2k-1}}} \bar\theta_{b_{i_{2k}}} \, ,
\eeq
where the sum runs over all possible fixed-point free involutions (also known as \emph{matchings} or \emph{pairings} in the literature) on $\{ 1, \ldots , 2m\}$, and we have furthermore adopted the convention that $i_{2k - 1} < i_{2k}$. We can then reorder the Grassmann variables appearing in each monomial to obtain
\beq
F_{\mathcal{B}} (\bar\theta) = \sum_{\mathrm{pairing} \, \sigma} (-1)^{{\mathrm cr}(\sigma)} \bar\theta_{b_1} \ldots \bar\theta_{b_m} \, ,
\eeq
where ${\mathrm cr}(\sigma)$ is the number of crossings in the fixed-point free involution $\sigma$. One may furthermore show, by induction on $m$, that  
\beq
\sum_{\mathrm{pairing} \, \sigma} (-1)^{{\mathrm cr}(\sigma)} = 1\,.
\eeq
Note that this formula can also be obtained as a particular case of the Touchard-Riordan formula (see \cite{t} and \cite{r}).
The numerator of \eqref{action} thus reduces to:
\beqa
\sum_{\mathcal{B} \subset \mathcal{I}, \vert \mathcal{B} \vert = 2m } \int [d\theta \cdot d \bar\theta ] \, [\bar\theta_\mathcal{B} \cdot \theta_\mathcal{A} ]  \, \exp\left(\sum_{i,j = 1}^n \bar\theta_i M^{-1}_{ij} \theta_j \right)\, .
\eeqa
Using now Lemma \ref{lemma:paths} completes the proof.
\end{proof}

Let us now consider the case in which some of the paths' endpoints are fixed. More precisely, for any $\mathcal{B} \subset V$ such that $\mathcal{B} \cap \mathcal{I} = \emptyset$, we define $\mathcal{F}_{{\cal A},{\cal B}}^{\mathcal{I}}$ as the set of flows from $\mathcal{A}$ to $\mathcal{B} \cup \mathcal{D}$, where $\mathcal{D}$ is an arbitrary subset of $\mathcal{I}$:
\beqa
\mathcal{F}_{{\cal A},{\cal B}}^{\mathcal{I}}:= \bigcup_{\mathcal{D} \subset \mathcal{I}
, \vert\mathcal{B} \cup \mathcal{D} \vert = \vert\mathcal{A}\vert} \mathcal{F}_{{\cal A},\mathcal{B} \cup \mathcal{D}}\,.
\eeqa
This is again a disjoint union of subsets; the functions $\mathrm{wt}$ and $\mathrm{sgn}$ are extended to $\mathcal{F}_{{\cal A},{\cal B}}^{\mathcal{I}}$ in a natural way.
For simplicity, we will only consider subsets $\mathcal{B}$ and $\mathcal{I}$ such that any element of $\mathcal{B}$ is strictly smaller than any element of $\mathcal{I}$, which we denote $\mathcal{B}<\mathcal{I}$. Note that this restriction does not affect the generality of Theorem \ref{thm32} below, since one is free to relabel the vertices of the graph in order to insure that $\mathcal{B}<\mathcal{I}$ holds. 

\begin{theorem}
\label{thm32}
Let $\mathcal{A}$, $\mathcal{B}$ and $\mathcal{I}$ be subsets of $\{ 1, \ldots , n\}$ such that: $\vert\mathcal{A}\vert = r$ and $\vert\mathcal{B}\vert = s$ with $r+s$ even, $s \leq r$, and $\mathcal{B} < \mathcal{I}$. One has: 
\beqa\label{formula:thm32}
\pf  
\begin{pmatrix}
Q^{\mathcal{I}}_{\mathcal{A}\mathcal{A}} & - M_{\mathcal{A}\mathcal{B}} \\
 {}^t M_{\mathcal{A}\mathcal{B}} & 0
\end{pmatrix}
 = 
\frac{\underset{(\mathbf{P}, \mathbf{C})\in \mathcal{F}_{{\cal A},{\cal B}}^{\mathcal{I}}}{\sum} \mathrm{sgn}(\mathbf{P})\mathrm{wt}(\mathbf{P}) \, \mathrm{sgn}(\mathbf{C})\mathrm{wt}(\mathbf{C})}{\underset{\mathbf{C} \in \mathcal{C}}{\sum} \mathrm{sgn}(\mathbf{C}) \mathrm{wt}(\mathbf{C})}\, .
\eeqa
\end{theorem}
\begin{proof}
We introduce the same matrix $S$ as in the proof of Theorem \ref{thm31}, and compute 
\begin{align}
\pf  
\begin{pmatrix}
Q^{\mathcal{I}}_{\mathcal{A}\mathcal{A}} & - M_{\mathcal{A}\mathcal{B}} \\
 {}^t M_{\mathcal{A}\mathcal{B}} & 0
\end{pmatrix}
 = \pf\left( S_{\mathcal{A}\cup \overline{\mathcal{B}} \mathcal{A}\cup \overline{\mathcal{B}}}\right)\,.
\end{align}
Using again Lemma \ref{wick} yields the same denominator as before, denominator which expands as a sum over nonintersecting cycles. The numerator can be written in the form
\beq
\sum_{\mathcal{D} \subset \mathcal{I}, 
\vert \mathcal{D} \vert = r-s } (-1)^{\frac{r-s}{2}} \int [d\theta \cdot d \bar\theta ] \, \theta_{a_1} \ldots \theta_{a_r} \, \bar\theta_{b_1} \ldots \bar\theta_{b_s} \, F_{\mathcal{D}} (\bar\theta) \, \exp\left(\sum_{i,j = 1}^n \bar\theta_i M^{-1}_{ij} \theta_j \right)\,,\nonumber
\eeq
where $F_{\mathcal{D}} (\bar\theta) = \bar\theta_{d_1} \ldots \bar\theta_{d_{r-s}}$ and $d_1 < \ldots < d_{r-s}$ are the elements of $\mathcal{D}$. Since $\mathcal{B} < \mathcal{D}$ for any $\mathcal{D}$ appearing in the previous formula, no extra sign needs to be included. The proof then goes along the lines of the proof of Theorem \ref{thm31}.
\end{proof}

\section{A generalization}

In this section we provide a generalization of the results of the previous sections. This generalization is, up to our knowledge, not existing in the combinatorics literature, even for acyclic digraphs.

Let us now consider two subsets $\mathcal{I}, \mathcal{J}\subset V$ such that $\mathcal{A}<\mathcal{I}$ and $\mathcal{B} < \mathcal{J}$. We denote by $\mathcal{F}_{{\cal A},{\cal B}}^{\mathcal{I},\mathcal{J}}$ the set of self-avoiding flows which respect the following conditions:
for each vertex $a\in\cal A$, there exists a path starting at $a$;
for each vertex $b\in \cal B$, there exists a path ending at $b$;
 and, finally, all other paths' starting points (resp. endpoints) are in $\mathcal{I}$ (resp. in $\mathcal{J}$). 
This means that we can define
$\mathcal{F}_{{\cal A},{\cal B}}^{\mathcal{I},\mathcal{J}}$ by the disjoint union:
\beqa
\mathcal{F}_{{\cal A},{\cal B}}^{\mathcal{I},\mathcal{J}}:= \bigcup_{\mathcal{A}' \subset \mathcal{I}, 
\mathcal{B}' \subset \mathcal{J}, 
\vert\mathcal{A} \cup \mathcal{A}' \vert = \vert\mathcal{B} \cup \mathcal{B}' \vert } \mathcal{F}_{{\cal A}\cup\mathcal{A'},\mathcal{B} \cup \mathcal{B}'}\,.
\eeqa
Once again, the sign and weight functions are naturally extended to $\mathcal{F}_{{\cal A},{\cal B}}^{\mathcal{I},\mathcal{J}}$. 
Following the notation of the previous section, for any $\mathcal{K} \subset V$ we denote by $B^\mathcal{K}$ the $n \times n$ matrix such that: $B^\mathcal{K}_{ij}=1$ if $i, j\in \mathcal{K}$ and $i<j$, $B^\mathcal{K}_{ij}=-1$ if $i, j\in \mathcal{I}$ and $ i> j$, and $B^\mathcal{K}_{ij} = 0$ otherwise.

We can now state the main result of this section:
\begin{theorem}
\label{thm41}
Let $\mathcal{A}$, $\mathcal{B}$, $\mathcal{I}$ and $\mathcal{J}$ be subsets of $\{ 1, \ldots , n\}$ such that: $\vert\mathcal{A}\vert = r$, $\vert\mathcal{B}\vert = s$ with $r+s$ even,  
$\mathcal{A} < \mathcal{I}$ and $\mathcal{B} < \mathcal{J}$. 
One has
\beqa
\pf \begin{pmatrix}
P_{\mathcal{A}\mathcal{A}} & - R_{\mathcal{A}\mathcal{B}} \\
{}^t R_{\mathcal{A}\mathcal{B}} & Q_{\mathcal{B}\mathcal{B}} 
\end{pmatrix} = 
\frac{\underset{(\mathbf{P}, \mathbf{C})\in \mathcal{F}_{{\cal A},{\cal B}}^{\mathcal{I},\mathcal{J}}}{\sum} \mathrm{sgn}(\mathbf{P})\mathrm{wt}(\mathbf{P}) \, \mathrm{sgn}(\mathbf{C})\mathrm{wt}(\mathbf{C})}{\underset{(\mathbf{P}, \mathbf{C})\in \mathcal{F}_{\emptyset,\emptyset}^{\mathcal{I},\mathcal{J}}}{\sum} \mathrm{sgn}(\mathbf{P})\mathrm{wt}(\mathbf{P}) \, \mathrm{sgn}(\mathbf{C})\mathrm{wt}(\mathbf{C})}\,,
\eeqa
where the $n\times n$ matrices $P,Q, R$ are defined by
\begin{align}
R := M \left( 1 + B^{\mathcal{J}} {}^t M B^{\mathcal{I}} M \right)^{-1}\,, \qquad P := M B^{\mathcal{J}} {}^t R\,, \qquad Q := {}^t M B^{\mathcal{I}} R\,.
\end{align} 
\end{theorem}
\begin{proof}
Let us give here the main idea of the proof.
We use again Lemma \ref{wick}, but this time we take: 
\beqa\label{s41}
S^{-1}=
\begin{pmatrix}
B_\mathcal{I} & {}^t M^{-1} \\
- M^{-1} & B_\mathcal{J}
\end{pmatrix}\,,\qquad 
S=
\begin{pmatrix}
P & - R \\
{}^t R & Q
\end{pmatrix}.
\eeqa
The proof then goes along the lines of the proof of Theorem \ref{thm31}.
\end{proof}

\begin{corollary}
\label{coro}
Making the same assumptions as in Theorem \ref{thm41} and further assuming that there exists no path connecting $\mathcal{I}$ to $\mathcal{J}$, one has:
\beqa
\pf 
\begin{pmatrix}
Q^\mathcal{J}_{\mathcal{A} \mathcal{A}} & - M_{\mathcal{A} \mathcal{B}} \\
{}^t M_{\mathcal{A} \mathcal{B}} & {}^t Q^{\mathcal{I}}_{\mathcal{B} \mathcal{B}} 
\end{pmatrix} = 
\frac{\underset{(\mathbf{P}, \mathbf{C})\in \mathcal{F}_{{\cal A},{\cal B}}^{\mathcal{I},\mathcal{J}}}{\sum} \mathrm{sgn}(\mathbf{P})\mathrm{wt}(\mathbf{P}) \, \mathrm{sgn}(\mathbf{C})\mathrm{wt}(\mathbf{C})}{\underset{\mathbf{C}\in \mathcal{C}}{\sum} \mathrm{sgn}(\mathbf{C})\mathrm{wt}(\mathbf{C})}\,.
\eeqa
\end{corollary}
\begin{proof}
We remark that $B^{\mathcal{I}} M B^{\mathcal{J}}= 0$. Hence the matrix $S$ 
given in \eqref{s41} 
reduces to: 
\beqa
S=
\begin{pmatrix}
Q^\mathcal{J} & - M \\
{}^t M & {}^t Q^{\mathcal{I}} 
\end{pmatrix}.
\eeqa
The result follows since $\mathcal{F}_{\emptyset,\emptyset}^{\mathcal{I},\mathcal{J}} = \mathcal{F}_{\emptyset,\emptyset}$, and is therefore identified with $\mathcal{C}$.
\end{proof}

%

\acknowledgements{The authors warmly thank Xavier Viennot for suggesting them to tackle this topic using Grassmann algebra tools, and Thomas Krajewski for stimulating discussions. \\
A. T. is partially supported by the ANR JCJC CombPhysMat2Tens and PN 09 37
01 02 grants.
This research was supported in part by Perimeter Institute for Theoretical Physics. Research at Perimeter Institute is supported by the Government of Canada through the Department of Innovation, Science and Economic Development Canada and by the Province of Ontario through the Ministry of Research, Innovation and Science.

\printbibliography

\end{document}